\documentclass[12pt]{amsproc}
\makeatletter
\def\@@and{}
\makeatother
\usepackage[utf8]{inputenc}
\usepackage{amsmath,amsfonts,amssymb,amsthm,mathtools}
\usepackage{xcolor}
\usepackage{graphicx}
\usepackage{tikz}
\usepackage{pgfplots}
\pgfplotsset{compat=1.18}
\usepackage{fullpage}
\usepackage{hyperref}
\hypersetup{
    colorlinks=true,
    citecolor=cyan,
    linkcolor=blue,
    filecolor=magenta,
    urlcolor=cyan,
    pdftitle={Linear response for random systems with a cusp},
    pdfauthor={Karim Rakhimov, Marks Ruziboev, Davrbek Oltiboev}
}

\newtheorem{theorem}{Theorem}[section]
\newtheorem{lemma}[theorem]{Lemma}

\newtheorem{proposition}[theorem]{Proposition}

\theoremstyle{definition}
\newtheorem{definition}[theorem]{Definition}

\theoremstyle{remark}

\numberwithin{equation}{section}
\newcommand{\abs}[1]{\lvert#1\rvert}
\newcommand{\norm}[1]{\lVert#1\rVert}

\newcommand{\R}{\mathbb R}
\renewcommand{\P}{\mathbb P}
\DeclareMathOperator{\Supp}{Supp}
\DeclareMathOperator{\spec}{spec}

\title{Linear response for random systems with a cusp}

\author[D.~Oltiboev]{Davrbek Oltiboev}
\address{V.I. Romanovskiy Institute of Mathematics, Uzbekistan Academy of Sciences}
\address{Tashkent, Uzbekistan}
\email{davrbek.oltiboyev@gmail.com}

\author[K.~Rakhimov]{Karim Rakhimov}
\address{V.I. Romanovskiy Institute of Mathematics, Uzbekistan Academy of Sciences}
\address{Tashkent, Uzbekistan}
\email{karimjon1705@gmail.com}

\author[M.~Ruziboev]{Marks Ruziboev}
\address{V.I. Romanovskiy Institute of Mathematics, Uzbekistan Academy of Sciences}
\address{Tashkent, Uzbekistan}
\email{marx.ruziboev@gmail.com}

\date{}

\begin{document}

\begin{abstract}
We study i.i.d.\ random compositions of cusp tent-like interval maps having a common cusp point and a common cusp value. For cusp exponents
$-1<\beta<-\frac12$, we prove that the associated annealed transfer operators have a spectral gap on $W^{1,1}(I)$ and $W^{2,1}(I)$. For all sufficiently small perturbations of the probability law, there is a stationary density $h_\varepsilon\in W^{2,1}(I)$, unique among stationary densities belonging to $W^{1,1}(I)$. Moreover, the map $\varepsilon\mapsto h_\varepsilon$ is differentiable at $\varepsilon=0$ in $W^{1,1}(I)$, and we obtain an explicit linear response formula.
\end{abstract}
\maketitle

\section{Introduction}
Invariant and stationary measures play a key role in the statistical study of
chaotic dynamical systems. For a large class of one-dimensional systems the
relevant invariant measure is absolutely continuous with respect to Lebesgue
measure, so that, by Birkhoff's ergodic theorem, it describes the asymptotic
behaviour of the orbits of a set of initial conditions of positive Lebesgue
measure. A question, which is interesting from  both theoretical and applied interest is the stability of
invariant ergodic measures under perturbations of the underlying dynamical system. When the
invariant density depends differentiably on the perturbation parameter in a
suitable topology, the system is said to admit \emph{linear response}, and the
derivative of the density is the response of the system to an infinitesimal
change of the dynamics.

The first results in this direction go back to Ruelle \cite{Ruelle}, who proved
differentiability of the SRB state for smooth uniformly hyperbolic systems, and
to Dolgopyat \cite{Dolgo}, who treated a class of partially hyperbolic systems.
Abstract functional-analytic framework was subsequently developed, in which the
transfer operator is made quasi-compact on a Banach space adapted to the
dynamics and the perturbed spectral data are controlled by the spectral
stability theory of Keller and Liverani from \cite{KL}. For the results in this direction, see \cite{Baladi, GL} and the
survey \cite{Balorelse}. For one-dimensional maps, linear response is by now
understood in several regimes. For piecewise expanding unimodal maps
\cite{BalSma08}, for smooth unimodal maps along transversal families
\cite{BalSma1, BalBenSchn}, for intermittent maps with an indifferent fixed
point \cite{BalTodd, Kor, BS, Lep}, and in connection with periodic-point
expansions and rigorous computation \cite{PolVit, BGNN}. Situations in which the
regularity of the response is limited by the observable rather than by the
dynamics are considered in \cite{BalKun}.

Stationary measures describe the statistical behaviour of random dynamical systems. A natural question is how such measures change when the probability law selecting the maps is perturbed. When the stationary density varies differentiably in a suitable topology, the system is said to admit linear response; see, in the random setting, \cite{BRS}.

Linear response is, however, not a generic phenomenon. Baladi \cite{Bal2007}
showed that the susceptibility function of a piecewise expanding interval map
need not admit the analytic continuation required by the response formula, and
Baladi and Smania \cite{BalSma08} proved that, for families of tent-like maps
with bounded derivatives, differentiability of $\varepsilon\mapsto h_\varepsilon$
holds along perturbations tangent to the topological class but fails for typical
perturbations which move the value of the turning point. For smooth unimodal
maps along transversal families, only Whitney--H\"older regularity in the
parameter is available \cite{BalBenSchn}. We refer to \cite{Balorelse} for a
discussion of these obstructions.

Stationary measures play the role of invariant measures for random dynamical
systems, and the corresponding question is how a stationary measure reacts to a
perturbation of the data defining the random system: either the maps which are
composed, or the probability law according to which they are selected.
Randomness is in some respects a regularizing mechanism, since averaging over
the noise can smooth the dependence of the stationary measure on the parameter,
and linear response may hold in situations in which the deterministic statement
fails. Response formulas for maps subject to additive noise, allowing critical
points as well as contracting and expanding regions, were obtained by Galatolo
and Giulietti \cite{GG}; see also \cite{HM} for an early framework in this
spirit, \cite{ADF} for the associated optimal-response problems and
\cite{GaSe} for second-order response. Closest to the present paper is the work
of Bahsoun, Ruziboev and Saussol \cite{BRS}, who studied i.i.d.\ compositions of
interval maps chosen according to a distribution $\P$ and proved
differentiability of the absolutely continuous stationary measure under a
perturbation $\P\mapsto\P_\varepsilon$ of the law, together with an explicit
response formula; their results cover, in particular, random systems whose
annealed transfer operator need not have a spectral gap. Quenched versions of
the problem, in which one differentiates the equivariant measures of the
fibrewise dynamics, were developed in \cite{DS, RS, DGS, CN}, and extended to
random and sequential intermittent maps in \cite{DGTS}; see also \cite{Sedro}
for an abstract fixed-point regularity result underlying several of these
arguments.

In the present paper we combine the cusp mechanism of \cite{BG} with the
random-law perturbation mechanism of \cite{BRS}. We keep a compact family
$\{f_\alpha\}_{\alpha\in\Omega}$ of cusp tent-like maps fixed and perturb only
the probability measure on the parameter space. The cusp point $c$ and the cusp
value $a$ are common to all maps of the family. Consequently the inverse-branch
formula for $L_\alpha$ has the same support interval $[0,a]$ for every
parameter, and no moving-boundary term appears when the map parameter is
differentiated. Our main result, Theorem~\ref{thm:main-LR}, asserts that for all
sufficiently small perturbations of the law there is a stationary density
$h_\varepsilon\in W^{2,1}(I)$, unique among stationary densities belonging to
$W^{1,1}(I)$, that the map $\varepsilon\mapsto h_\varepsilon$ is differentiable
at $\varepsilon=0$ in $W^{1,1}(I)$, and that its derivative is given by an
explicit resolvent formula.
 
Our results develop the pioneering work \cite{BG} in the following ways. First, the admissible range of the cusp
exponent is enlarged. In \cite{BG} the standing assumption is
$\beta\in(-1,-3/4)$, which comes from requiring square integrability of the most
singular coefficient appearing in the second derivative of $L_\alpha\phi$. Here
that zero-order term is estimated instead by combining an $L^1$ bound on the
coefficient with the embedding $W^{1,1}(I)\hookrightarrow C^0(I)$, which yields
the larger range $\beta\in(-1,-1/2)$. Second, the topology in which the response
is obtained is stronger: \cite{BG} gives differentiability of
$\varepsilon\mapsto h_\varepsilon$ in $L^1(I)$, whereas we obtain it in
$W^{1,1}(I)$ and hence, by the same embedding, uniformly on $I$. On the other
hand, the perturbations considered in \cite{BG} move the value of the turning
point, and this is exactly the class of perturbations excluded by our standing
assumption that $a$ is common to the family; it is this restriction which
removes the boundary terms and makes the stronger topology accessible. In this
sense the two results are complementary rather than nested. We note finally that
the deterministic case is contained in our setting: taking $\P=\delta_0$ and
$\P_\varepsilon=\delta_\varepsilon$ reduces Theorem~\ref{thm:main-LR} to a
statement about the one-parameter family $\{f_\varepsilon\}$, and gives a
$W^{1,1}$-response for deterministic cusp families with fixed cusp point and
cusp value. This is carried out in Section~\ref{sec:example}. Also, notice that the results of this paper can be extended to $W^{1, p}$ spaces in the spirit of \cite{CHD}.  
 
The comparison with \cite{BRS} is of a different nature. We follow the
perturbative scheme introduced there: the perturbation acts on the law and not
on the maps, differentiability of $\varepsilon\mapsto\P_\varepsilon$ is
understood in the weak sense of Definition~\ref{def:order-one-1D}, and the
response is read off at the end from a resolvent identity, as in
\cite[Proposition~6.2]{BRS}. The class of systems and the function spaces are,
however, different. 
The systems treated in \cite{BRS} include uniformly expanding
families as well as induced non-uniformly expanding systems such as
Gauss--R\'enyi and Pomeau--Manneville maps. The distinction relevant
here is therefore not simply bounded versus unbounded derivatives,
but the common-cusp Sobolev structure and the fixed support interval
used in our argument.
The random maps treated in \cite{BRS} have bounded assumed to have full branch inducing schemes with uniformly bounded distortion. We notice that the construction of such inducing schemes is highly non-trivial task and it is not clear in the setting of the current paper it is not clear how to construct such inducing schemes.  In our situation the
derivatives $f_\alpha'$ blow up at the common cusp point, so that the
bounded-variation setting is not available; in return, the cusp produces uniform
Lasota--Yorke inequalities for the pairs $(W^{1,1}(I),L^2(I))$ and
$(W^{2,1}(I),W^{1,1}(I))$, and therefore a genuine spectral gap for $L_{\P}$ on
both Sobolev spaces. This is what allows the response to be obtained in
$W^{1,1}(I)$. A further difference is the role played by mixing: assumption
\textup{(A7)}, which is a topological mixing hypothesis on the skew product
restricted to $\Supp\P\times[0,a]$, replaces the covering-type conditions of
\cite{BRS} and is used to exclude non-trivial peripheral spectrum and to obtain
uniqueness of the stationary density in $W^{1,1}(I)$.
 
The proof proceeds as follows. We first prove uniform Lasota--Yorke inequalities
give quasi-compactness on $W^{1,1}(I)$ and $W^{2,1}(I)$. In the random setting,
assumption~\textup{(A7)} is converted into an open-set positivity property; this
gives uniqueness of the stationary density in $W^{1,1}(I)$ and excludes
non-trivial peripheral cycles. We then prove the continuity properties required
by the Keller--Liverani theorem \cite{KL} and establish genuine $W^{1,1}$-valued
differentiability of $\alpha\mapsto L_\alpha h$ for $h\in W^{2,1}(I)$. Finally,
the resolvent argument used in \cite[Proposition~6.2]{BRS} yields the linear
response formula in $W^{1,1}(I)$.
 
The paper is organised as follows. In Section~\ref{sec:setting} we introduce the
random cusp maps, the perturbations of the probability law, the standing
assumptions \textup{(A1)--(A7)}, and we state the main result. Section~3 is
devoted to its proof. In Section~\ref{sec:example} we exhibit an explicit family
satisfying all the assumptions, for which the theorem gives a non-trivial
response formula.

\section{Random cusp maps and statement of the main result}\label{sec:setting}

Let $I=[0,1]$ and let $m$ denote Lebesgue measure on $I$. Let $\Omega\subset\R$ be a compact interval, equipped with its Borel $\sigma$-algebra $\mathcal F$, and let $\P$ be a probability measure on $(\Omega,\mathcal F)$. For every $\alpha\in\Omega$, let $f_\alpha:I\to I$ be a cusp map. Parameter derivatives at the endpoints of $\Omega$ are understood in the one-sided sense. A representative admissible map is shown in Figure~\ref{fig:cusp-family}, and the concrete family is described in Section~\ref{sec:example}.

If $\omega=(\omega_0,\omega_1,\ldots)\in\Omega^{\mathbb N}$ is distributed according to $\P^{\mathbb N}$, the associated random orbit is
\[
 x,\quad f_{\omega_0}(x),\quad
 f_{\omega_1}\circ f_{\omega_0}(x),\quad \ldots,\quad
 f_{\omega_{n-1}}\circ\cdots\circ f_{\omega_0}(x),\quad \ldots.
\]
Let $\sigma$ be the one-sided shift on $\Omega^{\mathbb N}$. The corresponding skew product is
\[
 T(\omega,x)=(\sigma\omega,f_{\omega_0}(x)).
\]

For every $\alpha\in\Omega$, let
$L_\alpha:L^1(I)\to L^1(I)$ be the transfer operator of $f_\alpha$, defined by
\[
 \int_I(L_\alpha\phi)\,\psi\,dm
 =
 \int_I\phi\,(\psi\circ f_\alpha)\,dm,
 \qquad
 \phi\in L^1(I),\quad \psi\in L^\infty(I).
\]

A probability measure $\mu$ on $I$ is called stationary if
\[
 \mu(A)=\int_\Omega\mu(f_\alpha^{-1}(A))\,d\P(\alpha)
\]
for every Borel set $A\subset I$. The annealed transfer operator is
\[
 L_{\P}\phi
 :=
 \int_\Omega L_\alpha\phi\,d\P(\alpha).
\]
If $\mu=h\,dm$, then $\mu$ is stationary precisely when
\[
 L_{\P}h=h,
 \qquad h\ge0,
 \qquad \int_Ih\,dm=1.
\]
Thus, an absolutely continuous stationary measure is determined by a function $h\ge0$ satisfying $L_{\P}h=h$ and $\int_Ih\,dm=1$.

We work on the Sobolev spaces $W^{1,1}(I)$ and $W^{2,1}(I)$, equipped with
\[
 \norm{\phi}_{W^{k,1}}
 :=
 \sum_{j=0}^k\norm{\phi^{(j)}}_{L^1},
 \qquad k=1,2.
\]
For $k=1,2$, let
\[
 W^{k,1}_0(I)
 :=
 \left\{\phi\in W^{k,1}(I):\int_I\phi\,dm=0\right\}.
\]

We use throughout the one-dimensional embeddings $W^{1,1}(I)\hookrightarrow C^0(I)$, $W^{2,1}(I)\hookrightarrow C^1(I)$ and the compact embeddings $W^{1,1}(I)\Subset L^2(I)$, $W^{2,1}(I)\Subset W^{1,1}(I)$.

We now introduce the perturbations of the probability law and state the assumptions used throughout the paper.

Let $V\subset\R$ be a neighbourhood of $0$ and let
$\{\P_\varepsilon\}_{\varepsilon\in V}$ be probability measures on
$\Omega$ with $\P_0=\P$. Define
\[
 L_{\P_\varepsilon}\phi
 :=
 \int_\Omega L_\alpha\phi\,d\P_\varepsilon(\alpha).
\]

Fix $c\in(0,1)$, $a\in(0,1]$ and
$\beta\in(-1,-1/2)$. We assume:

\begin{itemize}
\item[(A1)]
For every $\alpha\in\Omega$, the restrictions
$f_{0,\alpha}:=f_\alpha|_{[0,c)}$ and
$f_{1,\alpha}:=f_\alpha|_{(c,1]}$ are one-to-one.

\item[(A2)]
For every $\alpha\in\Omega$,
$f_\alpha(0)=f_\alpha(1)=0$ and
\[
 f_\alpha(c)=\lim_{x\to c^-}f_\alpha(x)
 =\lim_{x\to c^+}f_\alpha(x)=a.
\]

\item[(A3)]
For every $\alpha\in\Omega$,
$f_\alpha|_{I\setminus\{c\}}\in C^3$.

\item[(A4)]
There exists $\theta>1$ such that
\[
 \inf_{\alpha\in\Omega}\inf_{x\ne c}\abs{f_\alpha'(x)}\ge\theta.
\]

\item[(A5)]
There exist constants $m_1,M_1,M_2,M_3>0$ such that, for every
$\alpha\in\Omega$ and $x\ne c$,
\[
 m_1\abs{x-c}^{\beta}
 \le
 \abs{f_\alpha'(x)}
 \le
 M_1\abs{x-c}^{\beta},
\]
\[
 \abs{f_\alpha''(x)}
 \le
 M_2\abs{x-c}^{\beta-1},
 \qquad
 \abs{f_\alpha'''(x)}
 \le
 M_3\abs{x-c}^{\beta-2}.
\]

\item[(A6)]
For $k=0,1,2,3$, with $f_\alpha^{(0)}:=f_\alpha$, the functions
$(\alpha,x)\mapsto f_\alpha^{(k)}(x)$ are continuous on
$\Omega\times(I\setminus\{c\})$. For $k=0,1,2$, the parameter derivatives
$\partial_\alpha f_\alpha^{(k)}(x)$ exist and are continuous there, the mixed derivatives commute in the form
\[
 \partial_x\partial_\alpha f_\alpha^{(k)}
 =
 \partial_\alpha f_\alpha^{(k+1)},
 \qquad k=0,1,
 \]
and there are constants $D_0,D_1,D_2>0$ such that
\begin{equation}\label{eq:parameter-bounds}
 \abs{\partial_\alpha f_\alpha^{(k)}(x)}
 \le
 D_k\abs{x-c}^{\beta-k+1},
 \qquad k=0,1,2.
\end{equation}
In particular, $\partial_\alpha f_\alpha(x)\to0$ as $x\to c$, consistently with the common cusp value $a$.

\item[(A7)]
Let $S:=\Supp\P$. The restricted skew product
\[
 T:S^{\mathbb N}\times[0,a]\longrightarrow
 S^{\mathbb N}\times[0,a],
 \qquad
 T(\omega,x)=(\sigma\omega,f_{\omega_0}(x)),
\]
is topologically mixing, where $S^{\mathbb N}$ carries the product of the relative topology on $S$.
\end{itemize}

\begin{definition}\label{def:order-one-1D}
We say that $\varepsilon\mapsto\P_\varepsilon$ is \emph{differentiable at $\varepsilon=0$} if there exists a finite signed Borel measure $\nu$ on $\Omega$ such that, for every $\varphi\in C^1(\Omega)$,
\begin{equation}\label{eq:order-one-1D}
 \left.
 \frac{d}{d\varepsilon}
 \int_\Omega\varphi(\alpha)\,d\P_\varepsilon(\alpha)
 \right|_{\varepsilon=0}
 =
 \int_\Omega\varphi'(\alpha)\,d\nu(\alpha).
\end{equation}
\end{definition}

Equivalently, for every $\varphi\in C^1(\Omega)$,
\begin{equation}\label{eq:order-one-remainder}
 \int_\Omega\varphi\,d(\P_\varepsilon-\P)
 =
 \varepsilon\int_\Omega\varphi'\,d\nu+o(\varepsilon).
\end{equation}

On $I\setminus\{c\}$ define
\begin{equation}\label{eq:A-B-def}
 A_\alpha
 :=
 -\frac{\partial_\alpha f_\alpha}{f_\alpha'},
 \qquad
 B_\alpha
 :=
 \frac{\partial_\alpha f_\alpha\,f_\alpha''}{(f_\alpha')^2}
 -
 \frac{\partial_\alpha f_\alpha'}{f_\alpha'}.
\end{equation}
The estimates below show that $A_\alpha$ and $B_\alpha$ are bounded near the cusp, so their values at $c$ are irrelevant.

We can now state the main result.

\begin{theorem}\label{thm:main-LR}
Assume \textup{(A1)--(A7)} and suppose that
$\varepsilon\mapsto\P_\varepsilon$ is differentiable at $0$. Then there exists $\varepsilon_0>0$, with $[-\varepsilon_0,\varepsilon_0]\subset V$, such that, for every
$|\varepsilon|\le\varepsilon_0$, the operator $L_{\P_\varepsilon}$ admits a stationary density
\[
 h_\varepsilon\in W^{2,1}(I),
 \qquad h_\varepsilon\ge0,
 \qquad \int_Ih_\varepsilon\,dm=1.
\]
This density is unique among stationary densities belonging to $W^{1,1}(I)$. Moreover, the map $\varepsilon\mapsto h_\varepsilon$ is differentiable at $\varepsilon=0$ in $W^{1,1}(I)$, and
\begin{equation}\label{eq:main-response}
 \left.
 \frac{d}{d\varepsilon}h_\varepsilon
 \right|_{\varepsilon=0}
 =
 (I-L_{\P})^{-1}q
 \qquad\text{in }W^{1,1}(I),
\end{equation}
where $(I-L_{\P})^{-1}$ is on $W^{1,1}_0(I)$ and
\begin{equation}\label{eq:q-theorem}
 q
 =
 \int_\Omega
 L_\alpha\bigl(A_\alpha h_0'+B_\alpha h_0\bigr)
 \,d\nu(\alpha)
 \in W^{1,1}_0(I).
\end{equation}
\end{theorem}

\section{Proof of Theorem~\ref{thm:main-LR}}

The proof is divided into three parts. We first obtain the Lasota--Yorke inequalities and study the unperturbed annealed operator. We then prove continuity and differentiability of the annealed operators. In the last part, Keller--Liverani spectral stability gives a uniform spectral gap for the perturbed operators, and the linear response formula follows from a direct resolvent argument.

\subsection{Spectral gap for the unperturbed annealed operator}

\subsubsection{Lasota--Yorke estimates}

For $i=0,1$, let
\[
 g_{i,\alpha}:(0,a)\longrightarrow(0,c)
 \quad\text{or}\quad(c,1)
\]
be the inverse branch of $f_{i,\alpha}$. For $x\in(0,a)$,
\begin{equation}\label{eq:pointwise-L}
 (L_\alpha\phi)(x)
 =
 \sum_{i=0}^1
 \frac{\phi(g_{i,\alpha}(x))}
 {\abs{f_\alpha'(g_{i,\alpha}(x))}},
\end{equation}
and $L_\alpha\phi(x)=0$ for $x\in(a,1]$.

\begin{lemma}\label{lem:LY-W11}
There are constants $C_1>0$ and
$\lambda_1:=\theta^{-1}<1$ such that, for every $\alpha\in\Omega$ and every
$\phi\in W^{1,1}(I)$,
\begin{equation}\label{eq:LY-W11}
 \norm{L_\alpha\phi}_{W^{1,1}}
 \le
 \lambda_1\norm{\phi}_{W^{1,1}}
 +C_1\norm{\phi}_{L^2}.
\end{equation}
Moreover, the operators $L_\alpha$ are uniformly bounded on $L^2(I)$.
\end{lemma}

\begin{proof}
Fix $\alpha\in\Omega$. The proof of \cite[Lemma~3]{BG}, with
$T_\varepsilon$ replaced by $f_\alpha$, gives
\begin{equation}\label{eq:first-derivative-operator}
 (L_\alpha\phi)'
 =
 L_\alpha\left(\frac{\phi'}{f_\alpha'}\right)
 -
 L_\alpha\left(\frac{f_\alpha''}{(f_\alpha')^2}\phi\right)
 \qquad\text{on }(0,a).
\end{equation}
The branchwise expression for $L_\alpha\phi$ tends to zero as
$x\uparrow a$, because
$|f_\alpha'(y)|^{-1}=O(|y-c|^{-\beta})\to0$ as $y\to c$.
Thus extension by zero to $(a,1]$ produces no boundary term.

Since $L_\alpha$ is an $L^1$ contraction and
$|f_\alpha'|\ge\theta$,
\[
 \left\|
 L_\alpha\left(\frac{\phi'}{f_\alpha'}\right)
 \right\|_{L^1}
 \le
 \theta^{-1}\|\phi'\|_{L^1}.
\]
The only integrability condition in the lower-order term of the proof of
\cite[Lemma~3]{BG} is the uniform $L^2$ boundedness of
\[
 K_{1,\alpha}
 :=
 \frac{|f_\alpha''|}{|f_\alpha'|^2}.
\]
By \textup{(A5)},
\[
 K_{1,\alpha}(x)
 \le
 C|x-c|^{-\beta-1}.
\]
Since $\beta<-1/2$,
\[
 \sup_{\alpha\in\Omega}
 \|K_{1,\alpha}\|_{L^2}<\infty.
\]
Hence Hölder's inequality gives
\[
 \|(L_\alpha\phi)'\|_{L^1}
 \le
 \theta^{-1}\|\phi'\|_{L^1}
 +C\|\phi\|_{L^2}.
\]
Combining this estimate with the $L^1$ contraction property and
$\|\phi\|_{L^1}\le\|\phi\|_{L^2}$, we obtain
\begin{align*}
 \|L_\alpha\phi\|_{W^{1,1}}
 &\le
 \|\phi\|_{L^1}
 +\lambda_1\|\phi'\|_{L^1}
 +C\|\phi\|_{L^2}\\
 &\le
 \lambda_1\|\phi\|_{W^{1,1}}
 +(C+1-\lambda_1)\|\phi\|_{L^2}.
\end{align*}
This proves \eqref{eq:LY-W11}.

Finally, the proof of \cite[Lemma~5]{BG} applies to each $f_\alpha$.
For completeness, if $\psi_{i,\alpha}$ denotes the $i$th summand in
\eqref{eq:pointwise-L}, then
\[
 \|\psi_{i,\alpha}\|_{L^2}^2
 =
 \int_{I_i}
 \frac{|\phi(y)|^2}{|f_\alpha'(y)|}\,dy
 \le
 \theta^{-1}\|\phi\mathbf1_{I_i}\|_{L^2}^2.
\]
Since there are two branches,
\[
 \|L_\alpha\phi\|_{L^2}^2
 \le
 2\theta^{-1}\|\phi\|_{L^2}^2,
\]
uniformly in $\alpha$. The general case follows by density.
\end{proof}

\begin{lemma}\label{lem:LY-W21}
There is a constant $C_2>0$ such that, for every $\alpha\in\Omega$ and every
$\phi\in W^{2,1}(I)$,
\begin{equation}\label{eq:LY-W21}
 \norm{L_\alpha\phi}_{W^{2,1}}
 \le
 \lambda_2\norm{\phi}_{W^{2,1}}
 +C_2\norm{\phi}_{W^{1,1}},
 \qquad
 \lambda_2:=\theta^{-2}<1.
\end{equation}
\end{lemma}

\begin{proof}
For a smooth function $\phi$, differentiating
\eqref{eq:first-derivative-operator} once more gives
\begin{align}
 (L_\alpha\phi)''
 ={}&
 L_\alpha\left(\frac{\phi''}{(f_\alpha')^2}\right)
 -
 3L_\alpha\left(
 \frac{f_\alpha''}{(f_\alpha')^3}\phi'
 \right)
 \notag\\
 &+
 L_\alpha\left(
 \left[
 3\frac{(f_\alpha'')^2}{(f_\alpha')^4}
 -
 \frac{f_\alpha'''}{(f_\alpha')^3}
 \right]\phi
 \right).
 \label{eq:second-derivative-operator}
\end{align}
This is the same formula used in the proof of
\cite[Lemma~4]{BG}.

We first verify the boundary traces. Let $y\to c$. Since
$\phi,\phi'\in L^\infty(I)$,
\[
 \frac{\phi(y)}{\abs{f_\alpha'(y)}}
 =O(\abs{y-c}^{-\beta})\longrightarrow0,
\]
while the two terms in the first derivative satisfy
\[
 \frac{\abs{\phi'(y)}}{\abs{f_\alpha'(y)}^2}
 =O(\abs{y-c}^{-2\beta})\longrightarrow0,
 \qquad
 \frac{\abs{f_\alpha''(y)\phi(y)}}
 {\abs{f_\alpha'(y)}^3}
 =O(\abs{y-c}^{-2\beta-1})\longrightarrow0.
\]
The last limit uses $\beta<-1/2$. Hence both $L_\alpha\phi$ and
$(L_\alpha\phi)'$ have zero trace at $a$, and extension by zero to
$(a,1]$ produces no Dirac boundary term in either the first or the second weak derivative.

Using the $L^1$ contraction property and $\abs{f_\alpha'}\ge\theta$, the leading term in \eqref{eq:second-derivative-operator} satisfies
\[
 \left\|
 L_\alpha\left(\frac{\phi''}{(f_\alpha')^2}\right)
 \right\|_{L^1}
 \le
 \theta^{-2}\norm{\phi''}_{L^1}.
\]
Moreover,
\[
 \frac{\abs{f_\alpha''(x)}}{\abs{f_\alpha'(x)}^3}
 \le
 C\abs{x-c}^{-2\beta-1}.
\]
Since $-2\beta-1>0$, this coefficient is uniformly bounded, and hence
\[
 \left\|
 \frac{f_\alpha''}{(f_\alpha')^3}\phi'
 \right\|_{L^1}
 \le C\norm{\phi'}_{L^1}.
\]

For the zero-order term, put
\[
 K_{3,\alpha}
 :=
 3\frac{\abs{f_\alpha''}^2}{\abs{f_\alpha'}^4}
 +
 \frac{\abs{f_\alpha'''}}{\abs{f_\alpha'}^3}.
\]
By \textup{(A5)},
\[
 K_{3,\alpha}(x)
 \le
 C\abs{x-c}^{-2\beta-2}.
\]
The condition $\beta<-1/2$ is exactly what is needed for this weight to belong to $L^1(I)$, uniformly in $\alpha$. Thus, using
$W^{1,1}(I)\hookrightarrow L^\infty(I)$,
\[
 \norm{K_{3,\alpha}\phi}_{L^1}
 \le
 \norm{K_{3,\alpha}}_{L^1}\norm{\phi}_{L^\infty}
 \le
 C\norm{\phi}_{W^{1,1}}.
\]
This is the only point where we modify the proof of
\cite[Lemma~4]{BG}: the most singular coefficient is estimated in
$L^1$, rather than in $L^2$.

Consequently,
\begin{equation}\label{eq:second-order-seminorm}
 \norm{(L_\alpha\phi)''}_{L^1}
 \le
 \theta^{-2}\norm{\phi''}_{L^1}
 +C\norm{\phi}_{W^{1,1}}.
\end{equation}
Combining \eqref{eq:second-order-seminorm} with the boundedness of
$L_\alpha$ on $W^{1,1}(I)$ gives
\begin{align*}
 \norm{L_\alpha\phi}_{W^{2,1}} \le
 \lambda_2\norm{\phi''}_{L^1}
 +C\norm{\phi}_{W^{1,1}}\le
 \lambda_2\norm{\phi}_{W^{2,1}}
 +C_2\norm{\phi}_{W^{1,1}}.
\end{align*}
The general case follows by density of smooth functions in $W^{2,1}(I)$.
\end{proof}

\begin{lemma}\label{lem:annealed-LY}
For every $\varepsilon\in V$,
\begin{align}
 \norm{L_{\P_\varepsilon}\phi}_{W^{1,1}}
 &\le
 \lambda_1\norm{\phi}_{W^{1,1}}
 +C_1\norm{\phi}_{L^2},
 \label{eq:annealed-W11}\\
 \norm{L_{\P_\varepsilon}\psi}_{W^{2,1}}
 &\le
 \lambda_2\norm{\psi}_{W^{2,1}}
 +C_2\norm{\psi}_{W^{1,1}}.
 \label{eq:annealed-W21}
\end{align}
The constants are independent of $\varepsilon$.
\end{lemma}

\begin{proof}
For smooth functions, the branch formulas and Assumption~\textup{(A6)} imply strong measurability of $\alpha\mapsto L_\alpha\phi$ in the relevant Sobolev space. The general case follows by density and the uniform estimates in Lemmas~\ref{lem:LY-W11} and \ref{lem:LY-W21}; these estimates also give Bochner integrability. Hence, by Minkowski's inequality,
\begin{align*}
 \norm{L_{\P_\varepsilon}\phi}_{W^{1,1}} \le \int_\Omega\norm{L_\alpha\phi}_{W^{1,1}}\,d\P_\varepsilon(\alpha)\le
 \lambda_1\norm{\phi}_{W^{1,1}}
 +C_1\norm{\phi}_{L^2}.
\end{align*}
The second estimate follows in the same way from
Lemma~\ref{lem:LY-W21}.
\end{proof}

\subsubsection{The unperturbed annealed operator}

\begin{lemma}\label{lem:open-set-positivity}
Let $U,V\subset[0,a]$ be nonempty relatively open intervals. Then there exists $n_0\ge1$ such that
\begin{equation}\label{eq:open-set-positivity}
 \int_V L_{\P}^{n}\mathbf1_U\,dm>0
 \qquad
 \text{for every }n\ge n_0.
\end{equation}
\end{lemma}

\begin{proof}
Let $S:=\Supp\P$. For $n\ge1$ and
$\underline\alpha=(\alpha_0,\ldots,\alpha_{n-1})\in S^n$, write
\[
 f_{\underline\alpha}^{\,k}
 :=
 f_{\alpha_{k-1}}\circ\cdots\circ f_{\alpha_0},
 \qquad 1\le k\le n,
\]
and let $f_{\underline\alpha}^{\,0}$ be the identity. Iterating the annealed operator and using transfer-operator duality gives
\begin{equation}\label{eq:hitting-identity}
 \int_V L_{\P}^{n}\mathbf1_U\,dm
 =
 \int_{S^n}
 m\left(U\cap(f_{\underline\alpha}^{\,n})^{-1}(V)\right)
 \,d\P^n(\underline\alpha).
\end{equation}

By Assumption~\textup{(A7)}, there exists $n_0=n_0(U,V)$ such that, for every $n\ge n_0$,
\[
 T^n(S^{\mathbb N}\times U)
 \cap(S^{\mathbb N}\times V)\ne\varnothing.
\]
Fix such an $n$. Then there exist
$\omega=(\omega_0,\omega_1,\ldots)\in S^{\mathbb N}$ and $x\in U$ such that
$f_{\omega_{n-1}}\circ\cdots\circ f_{\omega_0}(x)\in V$.
Put $\underline\alpha=(\omega_0,\ldots,\omega_{n-1})$.

The finite set
\[
 E_{\underline\alpha,n}
 :=
 \bigcup_{k=0}^{n-1}
 (f_{\underline\alpha}^{\,k})^{-1}(\{c\})
\]
contains all points whose first $n$ iterates meet the cusp. Since the maps have continuous extensions at $c$,
$U\cap(f_{\underline\alpha}^{\,n})^{-1}(V)$ is a nonempty open set. We may therefore choose
\[
 x_0\in
 U\cap(f_{\underline\alpha}^{\,n})^{-1}(V)
 \setminus E_{\underline\alpha,n}.
\]
There is a non-degenerate closed interval $J\subset U$, containing $x_0$ in its interior, such that
\begin{equation}\label{eq:J-properties}
 f_{\underline\alpha}^{\,n}(J)\subset V,
 \qquad
 \operatorname{dist}\left(f_{\underline\alpha}^{\,k}(J),\{c\}\right)>0,
 \quad 0\le k\le n-1.
\end{equation}

Assumption~\textup{(A6)} gives locally uniform parameter continuity away from the cusp. Applying it successively along the finite orbit of $J$, we obtain relatively open neighbourhoods
$O_j\subset S$ of $\alpha_j$ such that, for every
$\underline\gamma=(\gamma_0,\ldots,\gamma_{n-1})\in
O_0\times\cdots\times O_{n-1}$, all the intermediate images of $J$ remain a positive distance from $c$ and $f_{\underline\gamma}^{\,n}(J)\subset V$.
Consequently,
\[
 m\left(U\cap(f_{\underline\gamma}^{\,n})^{-1}(V)\right)
 \ge m(J)>0.
\]
Since every nonempty relatively open subset of $S=\Supp\P$ has positive $\P$-measure,
\[
 \P^n(O_0\times\cdots\times O_{n-1})>0.
\]
The conclusion now follows from \eqref{eq:hitting-identity}.
\end{proof}

\begin{lemma}\label{lem:peripheral-cyclicity}
Suppose that, on the complexification of $W^{1,1}(I)$,
\[
 L_{\P}u=\lambda u,
 \qquad
 u\ne0,
 \qquad
 |\lambda|=1.
\]
Then $\lambda^k$ is an eigenvalue of $L_{\P}$ for every integer
$k\ge1$.
\end{lemma}

\begin{proof}
This is the annealed version of the fully cyclic argument from
\cite[Proposition~3.5]{Baladi}. We give the short
argument needed here.

Put
\[
 h:=|u|,
 \qquad
 s(x):=
 \begin{cases}
 \dfrac{u(x)}{|u(x)|},&u(x)\ne0,\\[1ex]
 0,&u(x)=0.
 \end{cases}
\]
Positivity gives
\[
 |u|
 =
 |L_{\P}u|
 \le
 L_{\P}|u|.
\]
Since $L_{\P}$ preserves integrals, the two sides have the same
integral, and hence
\begin{equation}\label{eq:peripheral-modulus-fixed}
 L_{\P}h=h.
\end{equation}

For almost every $x\in(0,a)$, the inverse-branch representation gives
\begin{align*}
 h(x)
 &=
 \sum_{i=0}^1
 \int_\Omega
 h(g_{i,\alpha}(x))|g_{i,\alpha}'(x)|
 \,d\P(\alpha),\\
 \lambda h(x)s(x)
 &=
 \sum_{i=0}^1
 \int_\Omega
 h(g_{i,\alpha}(x))s(g_{i,\alpha}(x))
 |g_{i,\alpha}'(x)|
 \,d\P(\alpha).
\end{align*}

If $h(x)>0$, division by the first identity shows that
$\lambda s(x)$ is a convex average of complex numbers of modulus one,
namely the values $s(g_{i,\alpha}(x))$ corresponding to non-zero
weights. 

Since $|\lambda s(x)|=1$, the equality case in the triangle inequality implies
\begin{equation}\label{eq:cyclic-phase}
 s(g_{i,\alpha}(x))
 =
 \lambda s(x)
\end{equation}
for $\P$-almost every $\alpha$ and every branch for which
$h(g_{i,\alpha}(x))>0$.

If $h(x)=0$, then \eqref{eq:peripheral-modulus-fixed} and the
non-negativity of the inverse-branch weights imply
\[
 h(g_{i,\alpha}(x))=0
\]
for $\P$-almost every $\alpha$ and both branches.

For $k\ge1$, define
\[
 u_k(x):=
 \begin{cases}
 |u(x)|
 \left(\dfrac{u(x)}{|u(x)|}\right)^k,
 &u(x)\ne0,\\[2ex]
 0,&u(x)=0.
 \end{cases}
\]
The polar map
\[
 \Phi_k(z):=
 \begin{cases}
 |z|\left(z/|z|\right)^k,&z\ne0,\\
 0,&z=0,
 \end{cases}
\]
is globally Lipschitz on $\mathbb C$; for instance,
\[
 |\Phi_k(z)-\Phi_k(w)|
 \le
 (2k+1)|z-w|.
\]
The Sobolev chain rule therefore gives
\[
 u_k=\Phi_k\circ u\in W^{1,1}(I).
\]

Using \eqref{eq:cyclic-phase}, for almost every $x$ with $h(x)>0$ we
obtain
\begin{align*}
 L_{\P}u_k(x)
 &=
 \sum_{i=0}^1
 \int_\Omega
 h(g_{i,\alpha}(x))s(g_{i,\alpha}(x))^k
 |g_{i,\alpha}'(x)|
 \,d\P(\alpha)\\
 &=
 \lambda^k s(x)^k
 \sum_{i=0}^1
 \int_\Omega
 h(g_{i,\alpha}(x))|g_{i,\alpha}'(x)|
 \,d\P(\alpha)\\
 &=
 \lambda^k h(x)s(x)^k
 =
 \lambda^k u_k(x).
\end{align*}
If $h(x)=0$, both sides vanish by the observation above. Thus
\[
 L_{\P}u_k=\lambda^ku_k.
\]
Since $|u_k|=h\ne0$, the number $\lambda^k$ is an eigenvalue of
$L_{\P}$ for every $k\ge1$.
\end{proof}

\begin{proposition}\label{prop:unperturbed-system}
The operator $L_{\P}$ admits a unique stationary density $h_0\in W^{2,1}(I)$ among stationary densities belonging to $W^{1,1}(I)$. Moreover,
$L_{\P}$ has a spectral gap when acting on both $W^{1,1}(I)$ and
$W^{2,1}(I)$. In particular, there exists $C>0$ such that
\begin{equation}\label{eq:unperturbed-resolvent}
 \left\|(I-L_{\P})^{-1}\right\|_{W^{1,1}_0\to W^{1,1}}
 \le C.
\end{equation}
\end{proposition}

\begin{proof}
The compact embeddings
\[
 W^{1,1}(I)\Subset L^2(I),
 \qquad
 W^{2,1}(I)\Subset W^{1,1}(I),
\]
together with \eqref{eq:annealed-W11} and
\eqref{eq:annealed-W21}, imply by the Hennion criterion that
$L_{\P}$ is quasi-compact on both Sobolev spaces and that the
corresponding essential spectral radii are strictly smaller than one;
see \cite{Hennion,Baladi}.

Let $\ell(\phi):=\int_I\phi\,dm$. Since $L_{\P}$ preserves integrals,
$L_{\P}^{*}\ell=\ell$, and hence $1\in\spec(L_{\P})$.
On $W^{1,1}(I)$ this point lies outside the essential spectral disk
and is therefore an isolated eigenvalue of finite algebraic
multiplicity. Moreover, no spectral value can have modulus greater
than one. Indeed, any such spectral value would lie outside the
essential spectral disk and hence would be an eigenvalue, which is
impossible because $L_{\P}$ is an $L^1$ contraction. Thus the
spectral radius of $L_{\P}$ is one.

Choose a nonzero real function $u\in W^{1,1}(I)$ satisfying
$L_{\P}u=u$. Positivity gives
\[
 L_{\P}|u|\ge |L_{\P}u|=|u|.
\]
The two sides have the same integral, and therefore $L_{\P}|u|=|u|$.
Consequently,
\[
 h_0
 :=
 \frac{|u|}{\int_I|u|\,dm}
\]
is a stationary density in $W^{1,1}(I)$.

We next prove uniqueness in $W^{1,1}(I)$. Suppose that
\[
 L_{\P}u=u,
 \qquad
 \int_Iu\,dm=0.
\]
It is enough to consider real-valued $u$. As above,
$L_{\P}|u|=|u|$, so the positive and negative parts $u^+$ and $u^-$
are invariant. Since every image under $L_{\P}$ vanishes on
$(a,1]$, the same is true of $u$.

If $u\ne0$, then both $u^+$ and $u^-$ are nonzero. Since
$W^{1,1}(I)\hookrightarrow C^0(I)$, there exist nonempty relatively
open intervals $U,V\subset[0,a]$ and $\delta>0$ such that
\[
 u^+\ge\delta\mathbf1_U,
 \qquad
 u^-\ge\delta\mathbf1_V.
\]
Since $u^+=0$ on $V$, Lemma~\ref{lem:open-set-positivity} gives, for
all sufficiently large $n$,
\[
 0
 =
 \int_Vu^+\,dm
 =
 \int_VL_{\P}^nu^+\,dm
 \ge
 \delta\int_VL_{\P}^n\mathbf1_U\,dm
 >0,
\]
a contradiction. Hence every invariant function in $W^{1,1}(I)$ is
a scalar multiple of $h_0$. In particular, the stationary density is
unique in this space.

The same argument applies to every power of $L_{\P}$. Indeed, suppose
that
\[
 L_{\P}^{q}u=u
\]
for some $q\ge1$. Subtracting $\left(\int_Iu\,dm\right)h_0$ reduces the problem to the zero-average case. Taking real and
imaginary parts if necessary, we may again work with a real-valued
function. 

Since $L_{\P}^{q}$ is a positive Markov operator, the same
argument gives
\[
 L_{\P}^{q}|u|=|u|,
\]
and hence the positive and negative parts of $u$ are fixed by
$L_{\P}^{q}$. Moreover, $(L_{\P}^{q})^n=L_{\P}^{qn}$.

Lemma~\ref{lem:open-set-positivity}, applied at the times $qn$, gives
the same contradiction as above. Thus, for every $q\ge1$, every
function fixed by $L_{\P}^{q}$ is a scalar multiple of $h_0$.

We now exclude the remaining peripheral eigenvalues. Suppose that
\[
 L_{\P}u=\lambda u,
 \qquad
 u\ne0,
 \qquad
 |\lambda|=1.
\]
By Lemma~\ref{lem:peripheral-cyclicity}, $\lambda^k$ is an eigenvalue
of $L_{\P}$ for every $k\ge1$. Since $L_{\P}$ is quasi-compact and
has spectral radius one, it has only finitely many eigenvalues on the
unit circle. Hence the set
\[
 \{\lambda^k:k\ge1\}
\]
is finite, so $\lambda$ is a root of unity. Thus $\lambda^q=1$ for
some $q\ge1$, and consequently $L_{\P}^{q}u=u$.

By the preceding uniqueness result for the powers of $L_{\P}$, $u$
is a scalar multiple of $h_0$. Since $L_{\P}h_0=h_0$, the
eigenvalue equation gives $\lambda=1$. Thus $1$ is the only
peripheral eigenvalue.

It remains to exclude a nontrivial Jordan block at $1$. Otherwise,
since the eigenspace at $1$ is spanned by $h_0$, after rescaling
there would exist $v\in W^{1,1}(I)$ such that
\[
 (L_{\P}-I)v=h_0.
\]
Then $L_{\P}^nv=v+nh_0$. This contradicts the $L^1$ contraction property, since
\[
 \norm{v+nh_0}_{L^1}
 \ge n-\norm{v}_{L^1},
 \qquad
 \norm{L_{\P}^nv}_{L^1}\le\norm{v}_{L^1}.
\]
Therefore $1$ is algebraically simple.

Since
\[
 W^{1,1}(I)
 =
 \operatorname{span}\{h_0\}\oplus W^{1,1}_0(I)
\]
and both subspaces are invariant under $L_{\P}$, the spectrum of
$L_{\P}$ restricted to $W^{1,1}_0(I)$ is contained in a disk of
radius strictly smaller than one. Hence
\[
 (I-L_{\P})^{-1}
 =
 \sum_{n=0}^{\infty}L_{\P}^n
\]
converges in operator norm on $W^{1,1}_0(I)$ and proves
\eqref{eq:unperturbed-resolvent}.

Finally, consider $L_{\P}$ on $W^{2,1}(I)$. Since
$L_{\P}^{*}\ell=\ell$ also on this space, $1$ belongs to its
spectrum. Its essential spectral radius is smaller than one, so $1$
lies outside the essential spectral disk and is therefore an
isolated eigenvalue. Let
\[
 0\ne v\in W^{2,1}(I),
 \qquad
 L_{\P}v=v.
\]
Viewed as an element of $W^{1,1}(I)$, the function $v$ is a scalar
multiple of $h_0$. Consequently, $h_0\in W^{2,1}(I)$.

Every peripheral eigenvector and every associated generalized
eigenvector on $W^{2,1}(I)$ also lies in $W^{1,1}(I)$. Likewise, a
spectral value of modulus greater than one on $W^{2,1}(I)$ would have
an eigenvector in $W^{1,1}(I)$, which is impossible. The conclusions
already proved on $W^{1,1}(I)$ therefore show that $1$ is
algebraically simple on $W^{2,1}(I)$ and that there are no other
peripheral eigenvalues. Hence $L_{\P}$ has a spectral gap on
$W^{2,1}(I)$ as well.
\end{proof}

\subsection{Perturbation of the annealed operators}

\subsubsection{Continuity}

We first establish the continuity properties needed in the spectral perturbation argument.

\begin{lemma}\label{lem:annealed-continuity}
For every $\phi\in W^{1,1}(I)$,
\begin{equation}\label{eq:annealed-strong-continuity}
 L_{\P_\varepsilon}\phi
 \longrightarrow
 L_{\P}\phi
 \ \ \text{in} \ \ W^{1,1}(I)
\end{equation}
as $\varepsilon\to0$. Moreover,
\begin{equation}\label{eq:annealed-mixed-close}
 \sup_{\norm{\phi}_{W^{1,1}}\le1}
 \norm{(L_{\P_\varepsilon}-L_{\P})\phi}_{L^2}
 \longrightarrow0.
\end{equation}
For every fixed $w\in L^1(I)$, the map
$\alpha\mapsto L_\alpha w$ is continuous in $L^1(I)$.
\end{lemma}

\begin{proof}
The expansion \eqref{eq:order-one-remainder} gives
\[
 \int_\Omega\varphi\,d\P_\varepsilon
 \longrightarrow
 \int_\Omega\varphi\,d\P
\]
for every $\varphi\in C^1(\Omega)$. Let $\varphi\in C(\Omega)$ and
choose $\varphi_j\in C^1(\Omega)$ with
$\|\varphi_j-\varphi\|_\infty\to0$. Since the measures are
probabilities,
\begin{align*}
 \left|\int_\Omega\varphi\,d\P_\varepsilon
 -\int_\Omega\varphi\,d\P\right|
 \le{}&2\|\varphi-\varphi_j\|_\infty+\left|\int_\Omega\varphi_j\,d\P_\varepsilon
 -\int_\Omega\varphi_j\,d\P\right|.
\end{align*}
Letting first $\varepsilon\to0$ and then $j\to\infty$ gives
\begin{equation}\label{eq:weak-convergence}
 \int_\Omega\varphi\,d\P_\varepsilon
 \longrightarrow
 \int_\Omega\varphi\,d\P
 \ \ 
 \text{for every} \ \ \varphi\in C(\Omega).
\end{equation}

We next prove that, for every fixed $\phi\in W^{1,1}(I)$, the map
$\alpha\mapsto L_\alpha\phi$ is continuous in $W^{1,1}(I)$. Suppose
first that $\phi\in C^1(I)$. On $(0,a)$ the inverse-branch formula
and its derivative are
\begin{align}
 (L_\alpha\phi)(x)
 &=
 \sum_{i=0}^1
 \frac{\phi(g_{i,\alpha}(x))}
 {|f_\alpha'(g_{i,\alpha}(x))|},
 \label{eq:branch-continuity-value}\\
 (L_\alpha\phi)'(x)
 &=
 \sum_{i=0}^1
 \left[
 \frac{\phi'(g_{i,\alpha}(x))}
 {f_\alpha'(g_{i,\alpha}(x))
  |f_\alpha'(g_{i,\alpha}(x))|}
 -
 \frac{f_\alpha''(g_{i,\alpha}(x))
 \phi(g_{i,\alpha}(x))}
 {|f_\alpha'(g_{i,\alpha}(x))|^3}
 \right].
 \label{eq:branch-continuity-derivative}
\end{align}

For every $\delta>0$, on the interval $[0,a-\delta]$, the inverse branches and all the
coefficients in these formulas depend continuously on $\alpha$ by
\textup{(A6)}.

Only the tail near $a$ requires a uniform estimate. If $x$ is close
to $a$ and $y=g_{i,\alpha}(x)$, then \textup{(A5)} and $\beta>-1$
give
\begin{equation}\label{eq:cusp-inverse-comparison}
 a-x
 =
 \left|\int_y^c f_\alpha'(t)\,dt\right|
 \asymp
 |y-c|^{\beta+1},
\end{equation}
uniformly in $\alpha$ and in the branch. Hence the preimage of
$(a-\delta,a)$ is contained in
$|y-c|\le C\delta^{1/(\beta+1)}$. Put
$r=C\delta^{1/(\beta+1)}$. A change of variables on each branch
shows that the $L^1$-mass of
\eqref{eq:branch-continuity-value} on this tail is bounded by
$C\|\phi\|_\infty r$. The two terms in
\eqref{eq:branch-continuity-derivative} are bounded there, after the
same change of variables, by
\[
 C\|\phi'\|_\infty\int_0^r t^{-\beta}\,dt,
 \qquad
 C\|\phi\|_\infty\int_0^r t^{-\beta-1}\,dt.
\]
Both tend to zero as $r\to0$, because $\beta<0$. We may therefore
first restrict to a compact subinterval of $(0,a)$, use uniform
parameter continuity there, and then let the omitted tail shrink.
This proves continuity in $W^{1,1}(I)$ for $C^1$ functions.

For a general $\phi\in W^{1,1}(I)$, choose
$\phi_j\in C^1(I)$ with $\phi_j\to\phi$ in $W^{1,1}(I)$.
Lemma~\ref{lem:LY-W11} and the embedding
$W^{1,1}(I)\hookrightarrow L^2(I)$ give
\begin{equation}\label{eq:uniform-W11-bound}
 \sup_{\alpha\in\Omega}
 \norm{L_\alpha\psi}_{W^{1,1}}
 \le C\norm{\psi}_{W^{1,1}}.
\end{equation}
If $\alpha_n\to\alpha$, then
\begin{align*}
 \norm{L_{\alpha_n}\phi-L_\alpha\phi}_{W^{1,1}}
 \le{}&C\norm{\phi-\phi_j}_{W^{1,1}}
 +\norm{L_{\alpha_n}\phi_j-L_\alpha\phi_j}_{W^{1,1}}\\
 &+C\norm{\phi_j-\phi}_{W^{1,1}}.
\end{align*}
For fixed $j$ the middle term tends to zero, and then $j\to\infty$
proves parameter continuity for every $\phi\in W^{1,1}(I)$.

Fix such a $\phi$ and put $F(\alpha)=L_\alpha\phi$. Given
$\eta>0$, compactness of $\Omega$ and continuity of $F$ give points
$\alpha_1,\ldots,\alpha_N\in\Omega$ and a continuous partition of
unity $\chi_1,\ldots,\chi_N$ such that
\[
 \sup_{\alpha\in\Omega}
 \norm{F(\alpha)-\sum_{j=1}^N
 \chi_j(\alpha)F(\alpha_j)}_{W^{1,1}}<\eta.
\]
Integrating this approximation against $\P_\varepsilon$ and $\P$
gives
\begin{align*}
 \norm{L_{\P_\varepsilon}\phi-L_{\P}\phi}_{W^{1,1}}
 \le{}&2\eta
 +\sum_{j=1}^N
 \left|\int_\Omega\chi_j\,d(\P_\varepsilon-\P)\right|
 \norm{F(\alpha_j)}_{W^{1,1}}.
\end{align*}
Each integral in the finite sum tends to zero by
\eqref{eq:weak-convergence}. Since $\eta>0$ is arbitrary,
\eqref{eq:annealed-strong-continuity} follows.

We now prove the mixed-norm convergence. If $\alpha_n\to\alpha$, then
\begin{equation}\label{eq:deterministic-mixed-continuity}
 \sup_{\norm{\phi}_{W^{1,1}}\le1}
 \norm{(L_{\alpha_n}-L_\alpha)\phi}_{L^2}
 \longrightarrow0.
\end{equation}
Indeed, the unit ball of $W^{1,1}(I)$ is relatively compact in
$L^2(I)$, while the operators $L_\alpha$ are uniformly bounded on
$L^2(I)$. Given $\eta>0$, choose a finite $L^2$-net
$\phi_1,\ldots,\phi_N$ for that unit ball. For every $\phi$ in the
ball choose $j$ with $\|\phi-\phi_j\|_{L^2}<\eta$. Then
\[
 \norm{(L_{\alpha_n}-L_\alpha)\phi}_{L^2}
 \le C\eta
 +\norm{(L_{\alpha_n}-L_\alpha)\phi_j}_{L^2},
\]
and the last term tends to zero for each fixed $j$. This proves
\eqref{eq:deterministic-mixed-continuity}. Since $\Omega$ is compact,
the continuity is uniform in the base parameter. Hence we may choose
points $\alpha_1,\ldots,\alpha_N$ and a continuous partition of unity
$\chi_1,\ldots,\chi_N$ such that
\[
 \sup_{\alpha\in\Omega}
 \sup_{\norm{\phi}_{W^{1,1}}\le1}
 \norm{L_\alpha\phi-
 \sum_{j=1}^N\chi_j(\alpha)L_{\alpha_j}\phi}_{L^2}<\eta.
\]
It follows that
\begin{align*}
 &\sup_{\norm{\phi}_{W^{1,1}}\le1}
 \norm{(L_{\P_\varepsilon}-L_{\P})\phi}_{L^2}\le
 2\eta+
 \sum_{j=1}^N
 \left|\int_\Omega\chi_j\,d(\P_\varepsilon-\P)\right|
 \sup_{\norm{\phi}_{W^{1,1}}\le1}
 \norm{L_{\alpha_j}\phi}_{L^2}.
\end{align*}
The finite sum tends to zero by \eqref{eq:weak-convergence}. Letting
$\eta\to0$ proves \eqref{eq:annealed-mixed-close}.

Finally, for $w\in C(I)$ the inverse-branch formula, the fixed support
interval and the same tail argument show that
$\alpha\mapsto L_\alpha w$ is continuous in $L^1(I)$. Approximation
of an arbitrary $w\in L^1(I)$ by continuous functions, together with
the $L^1$ contraction property, proves the last assertion.
\end{proof}

\subsubsection{Differentiation}

\begin{lemma}\label{lem:parameter-derivative}
For every $h\in W^{2,1}(I)$, the map
$\alpha\mapsto L_\alpha h$ is differentiable as a
$W^{1,1}(I)$-valued map. Its derivative depends continuously on
$\alpha$ and satisfies
\begin{equation}\label{eq:operator-parameter-derivative}
 \partial_\alpha(L_\alpha h)
 =
 L_\alpha(A_\alpha h'+B_\alpha h).
\end{equation}
Moreover,
\begin{equation}\label{eq:parameter-derivative-bound}
 \sup_{\alpha\in\Omega}
 \norm{\partial_\alpha(L_\alpha h)}_{W^{1,1}}
 \le C\norm{h}_{W^{2,1}}.
\end{equation}
\end{lemma}

\begin{proof}
Put $u_\alpha:=\partial_\alpha f_\alpha$. From
\eqref{eq:A-B-def}, the commutation of mixed derivatives in
\textup{(A6)}, and \eqref{eq:parameter-bounds}, one obtains, with
$t=\abs{x-c}$,
\begin{equation}\label{eq:A-B-cusp-orders}
 A_\alpha=O(t),
 \qquad
 A_\alpha'=B_\alpha=O(1),
 \qquad
 B_\alpha'=O(t^{-1}),
\end{equation}
uniformly in $\alpha$. More explicitly,
\begin{equation}\label{eq:B-prime-formula}
 B_\alpha'
 =
 2\frac{u_\alpha'f_\alpha''}{(f_\alpha')^2}
 +\frac{u_\alpha f_\alpha'''}{(f_\alpha')^2}
 -2\frac{u_\alpha(f_\alpha'')^2}{(f_\alpha')^3}
 -\frac{u_\alpha''}{f_\alpha'},
\end{equation}
and every term on the right is $O(t^{-1})$.

Set
\[
 v_\alpha:=A_\alpha h'+B_\alpha h.
\]
Since $h\in W^{2,1}(I)$, both $h$ and $h'$ are bounded. Hence
$v_\alpha\in L^\infty(I)$. Moreover,
\[
 v_\alpha
 \in
 AC_{\mathrm{loc}}((0,c))
 \cap
 AC_{\mathrm{loc}}((c,1)),
\]
and its local classical derivative is
\begin{equation}\label{eq:v-alpha-prime}
 v_\alpha'
 =
 A_\alpha h''
 +
 2B_\alpha h'
 +
 B_\alpha'h.
\end{equation}
We do not claim that $v_\alpha$ belongs to $W^{1,1}(I)$.

For $i=0,1$, define on $(0,a)$
\[
 w_{i,\alpha}(x)
 :=
 v_\alpha(g_{i,\alpha}(x))
 |g_{i,\alpha}'(x)|.
\]
The function $w_{i,\alpha}$ is locally absolutely continuous on
$(0,a)$. Since
\[
 g_{i,\alpha}'(x)
 =
 \frac{1}{f_\alpha'(g_{i,\alpha}(x))},
 \qquad
 |g_{i,\alpha}'(x)|
 =
 \frac{1}{|f_\alpha'(g_{i,\alpha}(x))|},
\]
its classical derivative is
\begin{align}
 w_{i,\alpha}'(x)
 ={}&
 \frac{
 v_\alpha'(g_{i,\alpha}(x))
 }{
 f_\alpha'(g_{i,\alpha}(x))
 |f_\alpha'(g_{i,\alpha}(x))|
 }
 \notag\\
 &-
 \frac{
 f_\alpha''(g_{i,\alpha}(x))
 v_\alpha(g_{i,\alpha}(x))
 }{
 |f_\alpha'(g_{i,\alpha}(x))|^3
 }.
 \label{eq:branch-v-alpha-derivative}
\end{align}

A change of variables $x=f_\alpha(y)$ on the $i$th branch gives
\begin{align}
 \int_0^a|w_{i,\alpha}'(x)|\,dx
 \le{}&
 \int_{I_i}
 \frac{|v_\alpha'(y)|}{|f_\alpha'(y)|}\,dy
 \notag\\
 &+
 \int_{I_i}
 \frac{|f_\alpha''(y)|}{|f_\alpha'(y)|^2}
 |v_\alpha(y)|\,dy,
 \label{eq:branch-v-L1}
\end{align}
where $I_0=(0,c)$ and $I_1=(c,1)$.

The required integrability follows from
\eqref{eq:A-B-cusp-orders}. With $t=|y-c|$,
\[
 \frac{|A_\alpha|}{|f_\alpha'|}=O(t^{1-\beta}),
 \qquad
 \frac{|B_\alpha|}{|f_\alpha'|}=O(t^{-\beta}),
 \qquad
 \frac{|B_\alpha'|}{|f_\alpha'|}=O(t^{-\beta-1}),
\]
and
\[
 \frac{|f_\alpha''A_\alpha|}{|f_\alpha'|^2}=O(t^{-\beta}),
 \qquad
 \frac{|f_\alpha''B_\alpha|}{|f_\alpha'|^2}=O(t^{-\beta-1}).
\]
The weights $t^{1-\beta}$ and $t^{-\beta}$ are bounded near the cusp,
while $t^{-\beta-1}$ belongs to $L^1$ because $\beta<0$. Therefore
\[
 \int_0^a|w_{i,\alpha}'(x)|\,dx
 \le
 C\norm{h}_{W^{2,1}},
\]
uniformly in $\alpha$ and $i$. Also,
\[
 \int_0^a|w_{i,\alpha}(x)|\,dx
 =
 \int_{I_i}|v_\alpha(y)|\,dy
 <\infty.
\]
Thus $w_{i,\alpha}\in W^{1,1}(0,a)$. As $x\uparrow a$, one has $g_{i,\alpha}(x)\to c$, and hence
\[
 |w_{i,\alpha}(x)|
 \le
 \frac{\norm{v_\alpha}_{L^\infty}}
 {|f_\alpha'(g_{i,\alpha}(x))|}
 \longrightarrow0.
\]
Consequently, $w_{i,\alpha}$ has zero trace at $a$. Its extension by
zero to $(a,1]$ belongs to $W^{1,1}(I)$ and produces no Dirac mass at
$a$. Since $L_\alpha v_\alpha = w_{0,\alpha}+w_{1,\alpha}$, we conclude that $L_\alpha v_\alpha\in W^{1,1}(I)$.

Summing \eqref{eq:branch-v-alpha-derivative} over the two branches
gives, as an identity of weak derivatives on $I$,
\begin{equation}\label{eq:D-alpha-spatial-derivative}
 \bigl(L_\alpha v_\alpha\bigr)'
 =
 L_\alpha\left(\frac{v_\alpha'}{f_\alpha'}\right)
 -
 L_\alpha\left(
 \frac{f_\alpha''}{(f_\alpha')^2}v_\alpha
 \right).
\end{equation}
Moreover, \eqref{eq:branch-v-L1} and the preceding estimates yield
\begin{align}
 \norm{L_\alpha v_\alpha}_{W^{1,1}}
 \le C\Bigg(&
 \norm{v_\alpha}_{L^1}
 +
 \int_I
 \frac{|v_\alpha'|}{|f_\alpha'|}\,dm
 \notag\\
 &+
 \int_I
 \frac{|f_\alpha''|}{|f_\alpha'|^2}
 |v_\alpha|\,dm
 \Bigg)
 \le
 C\norm{h}_{W^{2,1}}.
 \label{eq:parameter-W11-estimate}
\end{align}

We next prove continuity in $\alpha$. The functions
\[
 v_\alpha,
 \qquad
 \frac{v_\alpha'}{f_\alpha'},
 \qquad
 \frac{f_\alpha''}{(f_\alpha')^2}v_\alpha
\]
depend continuously on $\alpha$ as $L^1(I)$-valued functions. This follows from pointwise parameter continuity away from $c$ and dominated convergence using the weights displayed above. For example, if $\alpha_n\to\alpha$, then
\[
 \norm{L_{\alpha_n}v_{\alpha_n}-L_\alpha v_\alpha}_{L^1}
 \le
 \norm{v_{\alpha_n}-v_\alpha}_{L^1}
 +\norm{(L_{\alpha_n}-L_\alpha)v_\alpha}_{L^1},
\]
and both terms tend to zero by the $L^1$ contraction property and the last assertion of Lemma~\ref{lem:annealed-continuity}. Applying the same estimate to the two terms in
\eqref{eq:D-alpha-spatial-derivative} shows that
\begin{equation}\label{eq:D-alpha-continuity}
 \alpha\longmapsto
 L_\alpha v_\alpha
 \ \text{is continuous in }W^{1,1}(I).
\end{equation}

It remains to identify this continuous candidate with the actual
$W^{1,1}$-derivative. Differentiating $f_\alpha(g_{i,\alpha}(x))=x$ with respect to $\alpha$ gives
\[
 \partial_\alpha g_{i,\alpha}(x)
 =A_\alpha(g_{i,\alpha}(x)).
\]
Differentiating with respect to $x$ and using
$A_\alpha'=B_\alpha$ gives
\[
 \partial_\alpha |g_{i,\alpha}'(x)|
 =B_\alpha(g_{i,\alpha}(x))
 |g_{i,\alpha}'(x)|.
\]
Therefore, for almost every $x\in(0,a)$,
\[
 \frac{\partial}{\partial\alpha}
 \left(
 h(g_{i,\alpha}(x))|g_{i,\alpha}'(x)|
 \right)
 =
 \bigl(A_\alpha h'+B_\alpha h\bigr)
 (g_{i,\alpha}(x))|g_{i,\alpha}'(x)|.
\]
Summing over the two branches gives the pointwise formula
\eqref{eq:operator-parameter-derivative}.

By \eqref{eq:D-alpha-continuity}, the map
\[
 \gamma\longmapsto
 L_\gamma(A_\gamma h'+B_\gamma h)
\]
is continuous from $\Omega$ to $W^{1,1}(I)$. Hence the fundamental
theorem of calculus, applied branchwise, gives
\begin{equation}\label{eq:operator-FTC}
 L_{\alpha+s}h-L_\alpha h
 =
 \int_0^s
 L_{\alpha+r}(A_{\alpha+r}h'+B_{\alpha+r}h)\,dr
 \qquad\text{in }W^{1,1}(I),
\end{equation}
whenever the segment joining $\alpha$ and $\alpha+s$ is contained in
$\Omega$; at the endpoints the identity is understood one-sidedly.
It follows that
\begin{align*}
 &\left\|
 \frac{L_{\alpha+s}h-L_\alpha h}{s}
 -L_\alpha(A_\alpha h'+B_\alpha h)
 \right\|_{W^{1,1}}\le
 \sup_{\substack{\gamma\in\Omega\\
 |\gamma-\alpha|\le|s|}}
 \left\|
 L_\gamma(A_\gamma h'+B_\gamma h)
 -L_\alpha(A_\alpha h'+B_\alpha h)
 \right\|_{W^{1,1}}
 \longrightarrow0.
\end{align*}

Thus $\alpha\mapsto L_\alpha h$ is differentiable in
$W^{1,1}(I)$, with derivative given by
\eqref{eq:operator-parameter-derivative}. Finally,
\eqref{eq:parameter-W11-estimate} gives
\eqref{eq:parameter-derivative-bound}.
\end{proof}

\begin{lemma}\label{lem:forcing-term}
The map $\varepsilon\mapsto L_{\P_\varepsilon}h_0$ is differentiable at $\varepsilon=0$ in $W^{1,1}(I)$, and
\begin{equation}\label{eq:forcing-term}
 \frac{L_{\P_\varepsilon}h_0-L_{\P}h_0}{\varepsilon}
 \longrightarrow q
 \ \ \text{in} \ \ W^{1,1}(I),
\end{equation}
where
\begin{equation}\label{eq:q-main}
 q
 :=
 \int_\Omega
 L_\alpha(A_\alpha h_0'+B_\alpha h_0)
 \,d\nu(\alpha)
 \in W^{1,1}_0(I).
\end{equation}
\end{lemma}

\begin{proof}
Set $F(\alpha):=L_\alpha h_0$. By
Proposition~\ref{prop:unperturbed-system}, $h_0\in W^{2,1}(I)$, and
Lemma~\ref{lem:parameter-derivative} shows that $F$ is continuously differentiable as a $W^{1,1}(I)$-valued function, with
\[
 F'(\alpha)
 =L_\alpha(A_\alpha h_0'+B_\alpha h_0).
\]

For $\varepsilon\ne0$, define
\[
 T_\varepsilon(\varphi)
 :=
 \frac1\varepsilon
 \int_\Omega\varphi\,d(\P_\varepsilon-\P),
 \qquad \varphi\in C^1(\Omega).
\]
By \eqref{eq:order-one-remainder}, $T_\varepsilon(\varphi)$ converges for every $\varphi\in C^1(\Omega)$. A standard local form of the uniform boundedness principle (obtained by the usual Baire-category argument) therefore gives $M>0$ and $\varepsilon_1>0$ such that
\begin{equation}\label{eq:scalar-uniform-bound}
 |T_\varepsilon(\varphi)|
 \le
 M\left(\norm{\varphi}_{\infty}
 +\norm{\varphi'}_{\infty}\right)
\end{equation}
whenever $0<|\varepsilon|<\varepsilon_1$.

We use a finite-dimensional approximation to pass from scalar to
$W^{1,1}$-valued differentiation. Since $F'$ is continuous on the compact interval $\Omega$, for every $\eta>0$ there is a continuous function
$G:\Omega\to W^{1,1}(I)$, taking values in the span of finitely many vectors, such that
\[
 \sup_{\alpha\in\Omega}
 \norm{F'(\alpha)-G(\alpha)}_{W^{1,1}}<\eta.
\]
Fix $\alpha_*\in\Omega$ and set
\[
 F_\eta(\alpha)
 :=F(\alpha_*)+\int_{\alpha_*}^{\alpha}G(t)\,dt.
\]
Then $F_\eta$ takes values in a finite-dimensional subspace,
$F_\eta'=G$, and
\begin{equation}\label{eq:finite-dimensional-approximation}
 \sup_{\alpha\in\Omega}
 \norm{F(\alpha)-F_\eta(\alpha)}_{W^{1,1}}
 +
 \sup_{\alpha\in\Omega}
 \norm{F'(\alpha)-F_\eta'(\alpha)}_{W^{1,1}}
 \le C_\Omega\eta.
\end{equation}
For $F_\eta$, Definition~\ref{def:order-one-1D} applies componentwise and yields
\[
 \frac1\varepsilon
 \int_\Omega F_\eta\,d(\P_\varepsilon-\P)
 \longrightarrow
 \int_\Omega F_\eta'\,d\nu
 \qquad\text{in }W^{1,1}(I).
\]
Testing \eqref{eq:scalar-uniform-bound} against norm-one continuous linear functionals on $W^{1,1}(I)$ and using
\eqref{eq:finite-dimensional-approximation}, we obtain
\[
 \norm{
 \frac1\varepsilon
 \int_\Omega(F-F_\eta)\,d(\P_\varepsilon-\P)
 }_{W^{1,1}}
 \le C\eta,
\]
while
\[
 \norm{
 \int_\Omega(F'-F_\eta')\,d\nu
 }_{W^{1,1}}
 \le |\nu|(\Omega)\eta.
\]
Letting first $\varepsilon\to0$ and then $\eta\to0$ gives
\[
 \frac1\varepsilon
 \int_\Omega F(\alpha)\,d(\P_\varepsilon-\P)(\alpha)
 \longrightarrow
 \int_\Omega F'(\alpha)\,d\nu(\alpha)
\]
in $W^{1,1}(I)$. This is precisely \eqref{eq:forcing-term} and
\eqref{eq:q-main}.

Finally, $\int_I F(\alpha)\,dm=\int_Ih_0\,dm=1$ for every $\alpha$.
Differentiating this constant identity gives
$\int_I F'(\alpha)\,dm=0$. Hence $\int_Iq\,dm=0$, so
$q\in W^{1,1}_0(I)$.
\end{proof}

\subsection{Perturbed spectral gap and linear response}

\begin{lemma}\label{lem:uniform-gap}
There exist $\varepsilon_0>0$, $C>0$ and $\rho\in(0,1)$ such that, for every $|\varepsilon|\le\varepsilon_0$, the operator
$L_{\P_\varepsilon}$ admits a stationary density
\[
 h_\varepsilon\in W^{2,1}(I),
 \qquad h_\varepsilon\ge0,
 \qquad \int_Ih_\varepsilon\,dm=1,
\]
which is unique among stationary densities in $W^{1,1}(I)$. Moreover, for $k=1,2$,
\begin{equation}\label{eq:uniform-gap}
 \norm{L_{\P_\varepsilon}^n\phi}_{W^{k,1}}
 \le
 C\rho^n\norm{\phi}_{W^{k,1}},
 \qquad
 \phi\in W^{k,1}_0(I),\quad n\ge0,
\end{equation}
with the same constants for all such $\varepsilon$.
\end{lemma}

\begin{proof}
Lemma~\ref{lem:annealed-LY} gives uniform Lasota--Yorke
inequalities. Moreover, the uniform $L^2$ bound of
Lemma~\ref{lem:LY-W11} implies, for some $M\ge1$ independent of
$\varepsilon$,
\[
 \norm{L_{\P_\varepsilon}^n}_{L^2\to L^2}
 \le M^n.
\]
Iterating \eqref{eq:annealed-W11}, we obtain
\begin{equation}\label{eq:iterated-W11-LY}
 \norm{L_{\P_\varepsilon}^n\phi}_{W^{1,1}}
 \le
 C\lambda_1^n\norm{\phi}_{W^{1,1}}
 +CM^n\norm{\phi}_{L^2}.
\end{equation}
Together with \eqref{eq:annealed-mixed-close}, these are the
hypotheses of the Keller--Liverani theorem \cite{KL} for the pair $(W^{1,1}(I),L^2(I))$.

By Proposition~\ref{prop:unperturbed-system}, the eigenvalue $1$ of
$L_{\P}$ is algebraically simple and the rest of its spectrum on
$W^{1,1}(I)$ is contained in a disk of radius strictly smaller than
one. Keller--Liverani spectral stability therefore implies that, for
all sufficiently small $|\varepsilon|$, there is a
one-dimensional spectral subspace in a fixed neighbourhood of $1$,
while the remaining spectrum is uniformly contained in a disk of
radius strictly smaller than one.

Since $L_{\P_\varepsilon}$ preserves integrals,
\[
 L_{\P_\varepsilon}^{*}\ell=\ell,
 \qquad
 \ell(\phi)=\int_I\phi\,dm,
\]
and hence $1$ belongs to the spectrum of every
$L_{\P_\varepsilon}$. Since the spectrum outside the spectral
subspace near $1$ is uniformly contained in a disk of radius smaller
than one, the unique spectral value in that subspace is exactly $1$.
Its total algebraic multiplicity is one.

Choose a nonzero real fixed vector
$u_\varepsilon\in W^{1,1}(I)$. Positivity gives
\[
 L_{\P_\varepsilon}|u_\varepsilon|
 \ge
 |L_{\P_\varepsilon}u_\varepsilon|
 =
 |u_\varepsilon|.
\]
The two sides have the same integral, and therefore $L_{\P_\varepsilon}|u_\varepsilon| =
 |u_\varepsilon|$.

Thus
\[
 h_\varepsilon
 :=
 \frac{|u_\varepsilon|}
 {\int_I|u_\varepsilon|\,dm}
\]
is a stationary density in $W^{1,1}(I)$. The one-dimensionality of
the eigenspace at $1$ also shows that this density is unique among
stationary densities in $W^{1,1}(I)$.

Put
\[
 L^2_0(I)
 :=
 \left\{
 \phi\in L^2(I):
 \int_I\phi\,dm=0
 \right\},
 \qquad
 T_\varepsilon
 :=
 L_{\P_\varepsilon}\big|_{W^{1,1}_0(I)}.
\]
Since $L_{\P_\varepsilon}$ preserves integrals, the operator
$T_\varepsilon$ maps $W^{1,1}_0(I)$ into itself, and its weak
extension maps $L^2_0(I)$ into itself.

The uniform Lasota--Yorke estimate
\eqref{eq:iterated-W11-LY}, the weak $L^2$ growth estimate and the
mixed-norm convergence \eqref{eq:annealed-mixed-close} remain valid
after restriction to the pair $\bigl(W^{1,1}_0(I),L^2_0(I)\bigr).$ By Proposition~\ref{prop:unperturbed-system}, $r(T_0)<1$. Choose $r_*\in(0,1)$ such that
\[
 \lambda_1<r_*,
 \qquad
 \spec(T_0)\subset\{z\in\mathbb C:|z|<r_*\}.
\]
The Keller--Liverani resolvent estimates
\cite[Theorem~1 and Corollary~1]{KL}, applied to the restricted
operators, imply that, after decreasing $\varepsilon_0$ if necessary,
\begin{equation}\label{eq:restricted-spectrum}
 \spec(T_\varepsilon)
 \subset
 \{z\in\mathbb C:|z|<r_*\}
\end{equation}
for every $|\varepsilon|\le\varepsilon_0$.

Fix $\rho_1\in(r_*,1)$. The same theorem gives the uniform resolvent bound
\begin{equation}\label{eq:uniform-zero-resolvent}
 \sup_{|\varepsilon|\le\varepsilon_0}
 \sup_{|z|=\rho_1}
 \norm{(zI-T_\varepsilon)^{-1}}_{W^{1,1}_0\to W^{1,1}_0}
 <\infty.
\end{equation}
Since the circle $|z|=\rho_1$ surrounds
$\spec(T_\varepsilon)$, the Dunford formula gives
\[
 T_\varepsilon^n
 =
 \frac{1}{2\pi i}
 \int_{|z|=\rho_1}
 z^n(zI-T_\varepsilon)^{-1}\,dz.
\]
Using \eqref{eq:uniform-zero-resolvent}, we obtain
\[
 \norm{T_\varepsilon^n}_{W^{1,1}_0\to W^{1,1}_0}
 \le
 C\rho_1^n,
 \qquad n\ge0,
\]
where $C$ is independent of sufficiently small $\varepsilon$.
Equivalently,
\begin{equation}\label{eq:uniform-W11-gap}
 \norm{L_{\P_\varepsilon}^n\phi}_{W^{1,1}}
 \le
 C\rho_1^n\norm{\phi}_{W^{1,1}},
 \qquad
 \phi\in W^{1,1}_0(I),\quad n\ge0.
\end{equation}

We next obtain a uniform $W^{1,1}$ bound for the stationary
densities. The first-order seminorm estimate from the proof of
Lemma~\ref{lem:LY-W11}, applied to
$L_{\P_\varepsilon}h_\varepsilon=h_\varepsilon$, gives
\[
 (1-\lambda_1)
 \norm{h_\varepsilon'}_{L^1}
 \le
 C\norm{h_\varepsilon}_{L^2}.
\]
Since $h_\varepsilon\ge0$ and
$\norm{h_\varepsilon}_{L^1}=1$, the one-dimensional Sobolev estimate
gives
\[
 \norm{h_\varepsilon}_{L^\infty}
 \le
 1+\norm{h_\varepsilon'}_{L^1}.
\]
Consequently,
\[
 \norm{h_\varepsilon}_{L^2}^2
 \le
 \norm{h_\varepsilon}_{L^\infty}
 \norm{h_\varepsilon}_{L^1}
 \le
 1+\norm{h_\varepsilon'}_{L^1}.
\]
Combining the preceding inequalities gives
\[
 (1-\lambda_1)
 \norm{h_\varepsilon'}_{L^1}
 \le
 C\left(
 1+\norm{h_\varepsilon'}_{L^1}
 \right)^{1/2}.
\]
After squaring, this is a quadratic inequality for
$\norm{h_\varepsilon'}_{L^1}$ and therefore yields
\begin{equation}\label{eq:uniform-h-W11}
 \sup_{|\varepsilon|\le\varepsilon_0}
 \norm{h_\varepsilon}_{W^{1,1}}
 <\infty.
\end{equation}

It remains to prove the $W^{2,1}$ regularity. By
\eqref{eq:annealed-W21} and the compact embedding $W^{2,1}(I)\Subset W^{1,1}(I)$, the operator $L_{\P_\varepsilon}$ is quasi-compact on
$W^{2,1}(I)$, with essential spectral radius at most
$\lambda_2<1$. Since it preserves integrals, $1$ belongs to its
spectrum. As $1$ lies outside the essential spectral disk, it is an
isolated eigenvalue. Let
\[
 0\ne w_\varepsilon\in W^{2,1}(I),
 \qquad
 L_{\P_\varepsilon}w_\varepsilon=w_\varepsilon.
\]
Viewed in $W^{1,1}(I)$, the function $w_\varepsilon$ belongs to the
one-dimensional eigenspace at $1$ and is therefore a nonzero scalar
multiple of $h_\varepsilon$. Hence $h_\varepsilon\in W^{2,1}(I)$.

Applying \eqref{eq:annealed-W21} to the fixed density gives
\[
 (1-\lambda_2)
 \norm{h_\varepsilon}_{W^{2,1}}
 \le
 C_2\norm{h_\varepsilon}_{W^{1,1}}.
\]
Together with \eqref{eq:uniform-h-W11}, this gives
\[
 \sup_{|\varepsilon|\le\varepsilon_0}
 \norm{h_\varepsilon}_{W^{2,1}}
 <\infty.
\]

Finally, let $\phi\in W^{2,1}_0(I)$. Iterating
\eqref{eq:annealed-W21} yields
\[
 \norm{L_{\P_\varepsilon}^n\phi}_{W^{2,1}}
 \le
 \lambda_2^n\norm{\phi}_{W^{2,1}}
 +
 C\sum_{j=0}^{n-1}
 \lambda_2^{n-1-j}
 \norm{L_{\P_\varepsilon}^j\phi}_{W^{1,1}}.
\]
The zero-average space is invariant, so
\eqref{eq:uniform-W11-gap} applies to every term in the sum. Thus
\[
 \norm{L_{\P_\varepsilon}^n\phi}_{W^{2,1}}
 \le
 \left(
 \lambda_2^n
 +
 C\sum_{j=0}^{n-1}
 \lambda_2^{n-1-j}\rho_1^j
 \right)
 \norm{\phi}_{W^{2,1}}.
\]
Choose $\rho\in\bigl(\max\{\lambda_2,\rho_1\},1\bigr)$. The elementary convolution estimate gives
\[
 \lambda_2^n
 +
 \sum_{j=0}^{n-1}
 \lambda_2^{n-1-j}\rho_1^j
 \le
 C\rho^n.
\]
Therefore,
\[
 \norm{L_{\P_\varepsilon}^n\phi}_{W^{2,1}}
 \le
 C\rho^n\norm{\phi}_{W^{2,1}},
 \qquad
 \phi\in W^{2,1}_0(I).
\]
After enlarging $C$ and $\rho$ if necessary,
\eqref{eq:uniform-gap} holds simultaneously for $k=1$ and $k=2$.
\end{proof}

\begin{proof}[Proof of Theorem~\ref{thm:main-LR}]
Lemma~\ref{lem:uniform-gap} proves the existence, regularity and uniqueness assertions. It also gives
\begin{equation}\label{eq:uniform-resolvent-bound}
 \sup_{|\varepsilon|\le\varepsilon_0}
 \norm{(I-L_{\P_\varepsilon})^{-1}}_{W^{1,1}_0\to W^{1,1}}
 <\infty.
\end{equation}
Indeed, on $W^{1,1}_0(I)$ the inverse is represented by the convergent series
$\sum_{n\ge0}L_{\P_\varepsilon}^n$.

Using
$L_{\P_\varepsilon}h_\varepsilon=h_\varepsilon$ and
$L_{\P}h_0=h_0$, we obtain
\[
 (I-L_{\P_\varepsilon})(h_\varepsilon-h_0)
 =
 (L_{\P_\varepsilon}-L_{\P})h_0.
\]
Since both densities have integral one,
$h_\varepsilon-h_0\in W^{1,1}_0(I)$, and hence
\begin{equation}\label{eq:difference-quotient-resolvent}
 \frac{h_\varepsilon-h_0}{\varepsilon}
 =
 (I-L_{\P_\varepsilon})^{-1}
 \frac{(L_{\P_\varepsilon}-L_{\P})h_0}{\varepsilon}.
\end{equation}
By Lemma~\ref{lem:forcing-term},
\[
 \frac{(L_{\P_\varepsilon}-L_{\P})h_0}{\varepsilon}
 \longrightarrow q
 \ \ \text{in} \ \ W^{1,1}_0(I).
\]
Moreover, the resolvent identity gives
\begin{align*}
 &(I-L_{\P_\varepsilon})^{-1}q
 -(I-L_{\P})^{-1}q=
 (I-L_{\P_\varepsilon})^{-1}
 (L_{\P_\varepsilon}-L_{\P})
 (I-L_{\P})^{-1}q.
\end{align*}
The function $(I-L_{\P})^{-1}q$ belongs to $W^{1,1}_0(I)$.
Lemma~\ref{lem:annealed-continuity} makes the middle factor tend to zero in $W^{1,1}(I)$, and
\eqref{eq:uniform-resolvent-bound} therefore gives
\[
 (I-L_{\P_\varepsilon})^{-1}q
 \longrightarrow
 (I-L_{\P})^{-1}q
 \ \ \text{in} \ \ W^{1,1}(I).
\]
Finally,
\begin{align*}
 &\frac{h_\varepsilon-h_0}{\varepsilon}
 -(I-L_{\P})^{-1}q=
 (I-L_{\P_\varepsilon})^{-1}
 \left[
 \frac{(L_{\P_\varepsilon}-L_{\P})h_0}{\varepsilon}-q
 \right]+
 \left[(I-L_{\P_\varepsilon})^{-1}
 -(I-L_{\P})^{-1}\right]q.
\end{align*}
The first term tends to zero by Lemma~\ref{lem:forcing-term} and
\eqref{eq:uniform-resolvent-bound}, while the second tends to zero by the preceding resolvent identity. Thus
\[
 \frac{h_\varepsilon-h_0}{\varepsilon}
 \longrightarrow
 (I-L_{\P})^{-1}q
 \ \ \text{in} \ \ W^{1,1}(I),
\]
which proves \eqref{eq:main-response}.
\end{proof}

\section{A concrete admissible family}\label{sec:example}

We finish by showing that the assumptions of Section~\ref{sec:setting}
are non-vacuous. Let
\[
 c=\frac12,
 \qquad
 a=1,
 \qquad
 -1<\beta<-\frac12,
 \qquad
 1<\vartheta<2,
\]
and put
\[
 K
 :=
 \left(1-\frac{\vartheta}{2}\right)
 (\beta+1)2^{\beta+1}.
\]
Since $\beta+1>0$ and $\vartheta<2$, one has $K>0$. Define
\begin{equation}\label{eq:example-f0}
 f_0(x)
 =
 \begin{cases}
 \displaystyle
 \vartheta x
 +
 \frac{K}{\beta+1}
 \left[
 2^{-(\beta+1)}
 -
 \left(\frac12-x\right)^{\beta+1}
 \right],
 &0\le x\le\frac12,\\[4mm]
 \displaystyle
 \vartheta(1-x)
 +
 \frac{K}{\beta+1}
 \left[
 2^{-(\beta+1)}
 -
 \left(x-\frac12\right)^{\beta+1}
 \right],
 &\frac12\le x\le1.
 \end{cases}
\end{equation}
The choice of $K$ gives
\[
 f_0(0)=f_0(1)=0
\]
and
\[
 f_0\left(\frac12\right)
 =
 \frac{\vartheta}{2}
 +
 \frac{K}{\beta+1}2^{-(\beta+1)}
 =
 \frac{\vartheta}{2}
 +
 1-\frac{\vartheta}{2}
 =1.
\]
For $x\ne\frac12$,
\[
 f_0'(x)
 =
 \begin{cases}
 \displaystyle
 \vartheta
 +
 K\left(\frac12-x\right)^\beta,
 &x<\frac12,\\[2mm]
 \displaystyle
 -\vartheta
 -
 K\left(x-\frac12\right)^\beta,
 &x>\frac12.
 \end{cases}
\]
Consequently,
\[
 |f_0'(x)|>\vartheta>1.
\]
Each branch is monotone and maps its half interval onto $I$.

Writing $t=|x-\frac12|$, one has
\[
 \frac{|f_0'(x)|}{t^\beta}
 =
 K+\vartheta t^{-\beta}.
\]
Since $0<t\le\frac12$, the right-hand side is bounded above and below
by positive constants. Therefore
\[
 |f_0'(x)|
 \asymp
 \left|x-\frac12\right|^\beta.
\]
Direct differentiation also gives
\[
 |f_0''(x)|
 =
 O\left(
 \left|x-\frac12\right|^{\beta-1}
 \right),
 \qquad
 |f_0'''(x)|
 =
 O\left(
 \left|x-\frac12\right|^{\beta-2}
 \right).
\]

\begin{figure}[t]
\centering
\begin{tikzpicture}[x=8cm,y=5cm,>=stealth]

\pgfmathsetmacro{\betaval}{-0.75}
\pgfmathsetmacro{\thetaval}{1.5}
\pgfmathsetmacro{\pval}{\betaval+1}
\pgfmathsetmacro{\Kval}{
 (1-\thetaval/2)*\pval*pow(2,\pval)
}

\draw[->] (0,0) -- (1.06,0) node[right] {$x$};
\draw[->] (0,0) -- (0,1.08) node[above] {$y$};

\draw[dashed,gray] (0.5,0) -- (0.5,1);
\draw[dashed,gray] (0,1) -- (0.5,1);

\draw[
 black,
 line width=0.9pt,
 domain=0:0.4999,
 samples=250,
 smooth,
 variable=\x
]
plot
({
 \x
},
{
 \thetaval*\x
 +
 \Kval/\pval*
 (
  pow(2,-\pval)
  -
  pow(0.5-\x,\pval)
 )
});

\draw[
 black,
 line width=0.9pt,
 domain=0.5001:1,
 samples=250,
 smooth,
 variable=\x
]
plot
({
 \x
},
{
 \thetaval*(1-\x)
 +
 \Kval/\pval*
 (
  pow(2,-\pval)
  -
  pow(\x-0.5,\pval)
 )
});

\fill (0,0) circle (0.55pt);
\fill (0.5,1) circle (0.65pt);
\fill (1,0) circle (0.55pt);

\node[below] at (0,0) {\small $0$};
\node[below] at (0.5,0) {\small $c=\frac12$};
\node[below] at (1,0) {\small $1$};
\node[left] at (0,1) {\small $a=1$};
\node at (0.77,0.72) {\small $f_0$};

\end{tikzpicture}
\caption{A representative admissible cusp map from
\eqref{eq:example-f0}, with $\beta=-\frac34$,
$\vartheta=\frac32$, common cusp point $c=\frac12$ and common cusp
value $a=1$. Its two branches are full and uniformly expanding.}
\label{fig:cusp-family}
\end{figure}
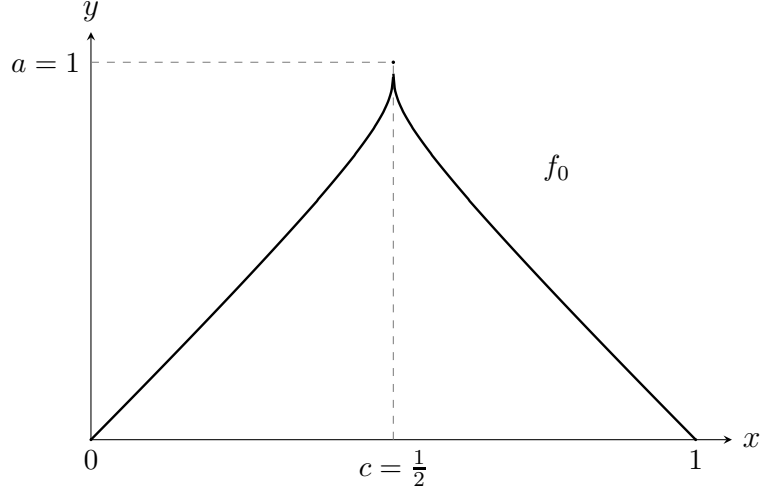

We now introduce a non-trivial parameter dependence. Define
\begin{equation}\label{eq:example-psi}
 \psi(x)
 =
 \begin{cases}
 \displaystyle
 x\left(\frac12-x\right)^{\beta+2},
 &0\le x<\frac12,\\[2mm]
 0,
 &x=\frac12,\\[2mm]
 \displaystyle
 (1-x)\left(x-\frac12\right)^{\beta+2},
 &\frac12<x\le1,
 \end{cases}
\end{equation}
and
\[
 f_\alpha
 :=
 f_0+\alpha\psi,
 \qquad
 \alpha\in[-\alpha_0,\alpha_0].
\]
Notice that
\[
 \psi(0)=\psi\left(\frac12\right)=\psi(1)=0.
\]
Moreover, as $x\to\frac12$,
\[
 \psi^{(k)}(x)
 =
 O\left(
 \left|x-\frac12\right|^{\beta+2-k}
 \right),
 \qquad
 k=0,1,2,3.
\]
In particular, for $k=0,1,2$,
\[
 \partial_\alpha f_\alpha^{(k)}
 =
 \psi^{(k)}
 =
 O\left(
 \left|x-\frac12\right|^{\beta+2-k}
 \right),
\]
which is stronger than the bound required in \textup{(A6)}.

The function $\psi'$ is bounded. Hence, by taking $\alpha_0>0$
sufficiently small, we have
\[
 \inf_{\substack{|\alpha|\le\alpha_0\\x\ne1/2}}
 |f_\alpha'(x)|
 \ge
 \vartheta-\alpha_0\norm{\psi'}_{L^\infty}
 >1.
\]
Thus the signs of the two branch derivatives do not change. Since the
endpoint and cusp values are fixed, every $f_\alpha$ maps $I$ into
$I$ and has two monotone full branches.

Near the cusp,
\[
 \frac{|\psi'(x)|}{|x-\frac12|^\beta}
 =
 O\left(|x-\tfrac12|\right),
\]
while
\[
 \psi''(x)
 =
 O\left(|x-\tfrac12|^\beta\right),
 \qquad
 \psi'''(x)
 =
 O\left(|x-\tfrac12|^{\beta-1}\right).
\]
It follows, after decreasing $\alpha_0$ if necessary, that the
two-sided derivative estimate and the second- and third-derivative
bounds in \textup{(A5)} hold uniformly for
$|\alpha|\le\alpha_0$. The continuity and mixed-derivative
requirements in \textup{(A6)} are immediate because the family
depends affinely on $\alpha$. Hence \textup{(A1)--(A6)} hold
uniformly on $\Omega=[-\alpha_0,\alpha_0]$.

The two inverse branches of $f_0$ have Lipschitz constants at most
$\vartheta^{-1}<1$. Consequently, every cylinder interval of rank
$n$ has diameter at most $\vartheta^{-n}$ and is mapped by $f_0^n$
onto $I$. Given a nonempty open interval $U\subset I$, choose a point
of $U$ which is not an endpoint of any cylinder. For all sufficiently
large $n$, the rank-$n$ cylinder containing this point is contained
in $U$. It follows that $f_0^n(U)=I$ for some $n$, so $f_0$ is topologically exact and, in particular,
topologically mixing on $I$.

Take $\P=\delta_0$. Then $\Supp\P=\{0\}$, and the restricted skew product in
\textup{(A7)} reduces to the dynamics of $f_0$. Therefore
\textup{(A7)} holds.

Finally, let
\[
 V=(-\alpha_0,\alpha_0),
 \qquad
 \P_\varepsilon=\delta_\varepsilon,
 \qquad
 \varepsilon\in V.
\]
For every $\varphi\in C^1(\Omega)$,
\[
 \left.
 \frac{d}{d\varepsilon}
 \int_\Omega
 \varphi(\alpha)\,d\P_\varepsilon(\alpha)
 \right|_{\varepsilon=0}
 =
 \left.
 \frac{d}{d\varepsilon}
 \varphi(\varepsilon)
 \right|_{\varepsilon=0}
 =
 \varphi'(0)
 =
 \int_\Omega
 \varphi'(\alpha)\,d\delta_0(\alpha).
\]
Thus the differentiability assumption on the probability laws holds
with $\nu=\delta_0$. This gives a non-trivial family to which
Theorem~\ref{thm:main-LR} applies.

\subsection{Acknowledgments.} DO has received funding from the Ministry of Higher Education, Science and Innovation of the Republic of Uzbekistan under State Grant No. FL9524115114.

\end{document}